\documentclass[oneside,a4paper]{amsart}

\usepackage{amsmath,amsthm,verbatim,amssymb,amsfonts,amscd, graphicx}
\usepackage{graphics}

\usepackage[utf8]{inputenc}
\usepackage[T1]{fontenc}

\usepackage{stmaryrd}
\usepackage{xspace}
\usepackage[all]{xy}
\usepackage[margin=2.5cm]{geometry}

\usepackage{blindtext} 
\usepackage[linktocpage]{hyperref}

\theoremstyle{plain}
\newtheorem{thm}{Theorem}[section]
\newtheorem{cor}[thm]{Corollary}
\newtheorem{lem}[thm]{Lemma}
\newtheorem{prop}[thm]{Proposition}

\theoremstyle{definition}

\newtheorem*{rmq}{Remark}

\newcommand{\ho}{\mathrm{H}}
\newcommand{\Hom}{\mathrm{Hom}}

\newcommand{\Oc}{\mathcal{O}}

\newcommand{\li}{\mathcal{L}}
\newcommand{\mi}{\mathcal{M}}

\newcommand{\Z}{\mathbb{Z}}

\newcommand{\bbC}{\mathbb{C}}
\newcommand{\pr}{\mathbb{P}}
\newcommand{\F}{\mathbb{F}}

\newcommand{\Ker}{\mathrm{Ker}}

\usepackage[shortlabels]{enumitem}
\setlist[enumerate]{nosep}

\begin{document}

\title{
	Smoothing of 1-cycles over finite fields
	}
\author{Xiaozong WANG}
\address{Morningside Center of Mathematics, Chinese Academy of Sciences, No.55, Zhongguancun East Road, Beijing, 100190, China}
\email{xiaozong.wang@amss.ac.cn}

\begin{abstract}
	Let $X$ be a smooth projective variety defined over a finite field. We show that any algebraic $1$-cycle on $X$ is rationally equivalent to a smooth $1$-cycle, which is a $\mathbb{Z}$-linear combination of smooth curves on $X$. We also prove a generalized version of Poonen's Bertini theorem over finite fields. Given a very ample line bundle $\mathcal{L}$ on $X$ and an arbitrary line bundle $\mathcal{M}$, this version implies the existence of a global section of $\mathcal{M}\otimes \mathcal{L}^{\otimes d}$ for sufficiently large $d$ whose divisor is smooth.
\end{abstract}

\maketitle


\section{Introduction}
\subsection{Main result}
	Let $X$ be a smooth projective variety defined over a field $k$. We say that an algebraic $d$-cycle $\sum_{i}a_iZ_i$ on $X$ is \emph{smooth} if each $Z_i$ is the class of a smooth subvariety of $X$ of dimension $d$. 
	
	Over a field of characteristic $0$, given an algebraic cycle of a smooth projective variety, the question on the existence of rationally equivalent smooth algebraic cycles is considered by Hironaka. In \cite{Hi68}, he proved that on a smooth projective variety $X$ of dimension $n$, every algebraic cycle of dimension at most $\min(3,\frac{n-1}{2})$ is rationally equivalent to a smooth one. In particular, it is true for $1$-cycles whenever $X$ is of dimension at least $3$. (The case when $\dim X=2$ can be proved directly via divisors.) His method relies on blowing-ups, a thorough study on the size of singular loci when pushing forward a cycle by a birational map, and also properties of general intersections of several hypersurfaces of $X$. In particular, his proof needs desingularization of subvarieties of $X$.
	
	In this paper, we prove the analogue of Hironaka's result on $1$-cycles over finite fields. Let $\F_q$ be a finite field of $q$ elements and of characteristic $p$. Our main result is the following theorem :
	\begin{thm}\label{1cycle}
		Let $X$ be a smooth projective variety over $\F_q$ of dimension $n\geq 2$. Then every algebraic $1$-cycle on $X$ is rationally equivalent to a smooth $1$-cycle, i.e. for any algebraic $1$-cycle $\sum_i a_iZ_i$ on $X$, where each $Z_i$ is the class of a closed curve on $X$ and $a_i\in \Z$, we can find another $1$-cycle $\sum_j b_jZ'_j$ such that 
		\[
			\sum_j b_jZ'_j\sim_{\mathrm{rat}} \sum_i a_iZ_i, 
		\]
		and that each $Z'_j$ is the class of a smooth closed curve on $X$.
	\end{thm}
	In the finite field case, we still follow Hironaka's idea to construct rationally equivalent $1$-cycles via intersections of divisors of sufficiently high degree. But as we only have finitely many divisors of any fixed degree, we cannot apply Hironaka's method of studying the property of the intersection of general hypersurfaces. Instead, we use Poonen's closed point sieve method to find divisors we need. Our main tools are a result of Poonen on the density of hypersurface sections containing a given subscheme in \cite{Po08}, and a generalization of his Bertini theorem over finite field in \cite{Po04}, which is stated below.

	For a scheme $X$ over $\F_q$, the zeta function $\zeta_X$ of $X$ is defined by the infinite product
	\[
		\zeta_X(s)=\prod_{x\in \lvert X\rvert} \left( 1-q^{-s\deg x} \right)^{-1}
	\]
	This product converges for any $s\in \bbC$ such that $\mathrm{Re}(s)>\dim X$.
	\begin{thm}\label{generalbertini}
	Let $X$ be a projective variety of dimension $n\geq 1$ over $\F_q$. Let $\li$ be a very ample line bundle on $X$ and $\mi$ an arbitrary line bundle on $X$. Assume that there exists a closed subscheme $Y$ of $X$ such that $X_0:=X-Y$ is smooth. Define 
	\[
		\mathcal{P}_{d}:=\left\{ \sigma\in \ho^0(X, \mi\otimes \li^{\otimes d})\ ;\ 
		\mathrm{div}\sigma\cap X_0 \text{ is smooth of dimension } n-1 
		\right\},
	\]
	and $\mathcal{P}=\bigcup_{d\geq 0}\mathcal{P}_{d}$. Then we have
	\[
		\mu(\mathcal{P}):=\lim_{d\rightarrow \infty}\frac{\#\mathcal{P}_{d}}{\#\ho^0(X, \mi\otimes \li^{\otimes d})}=\zeta_{X_0}(n+1)^{-1}.
	\]
\end{thm}
\begin{rmq}
	When $\mi$ is trivial, this is Theorem 1.1 in \cite{Po04} for the quasi-projective scheme $X_0$.
\end{rmq}

What we need in the proof of Theorem \ref{1cycle} is a variant of this theorem :
\begin{cor}\label{generalbertinisingularity}
	Let $X$ be a projective variety of dimension $n\geq 1$ over $\F_q$. Let $\li$ be a very ample line bundle on $X$ and $\mi$ an arbitrary invertible sheaf on $X$. Assume that there exists a closed subscheme $Y$ of $X$ such that $X_0:=X-Y$ is smooth. 
	For any finite subscheme $S$ of $Y$, define
		\[
			\mathcal{P}_{S,d}:=\left\{ \sigma\in \ho^0(X, \mi\otimes \li^{\otimes d})\ ;\ 
			\begin{array}{ll}
			\mathrm{div}\sigma\cap X_0 \text{ is smooth of dimension } n-1, \\
			\mathrm{div}\sigma\cap S=\emptyset
			\end{array}
			\right\},
		\]
		and $\mathcal{P}_S=\bigcup_{d\geq 0}\mathcal{P}_{S,d}$. Then we have
		\[
			\mu(\mathcal{P}_S)=\zeta_{X_0}(n+1)^{-1}\cdot\prod_{x\in \lvert S\rvert} \left( 1-\#\kappa(x)^{-1}\right).
		\]
		In particular, we have $\mu(\mathcal{P}_S)>0$, and $\mathcal{P}_{S,d}\not=\emptyset$ when $d\gg 1$.
\end{cor}

\subsection{Strategy of the proof}
	Let $Z=\sum_i a_iZ_i$ be an $1$-cycle on $X$, where $a_i\in \Z$ and $Z_i$ are closed curves on $X$. As in the proof of Hironaka in \cite{Hi68}, we choose a very ample line bundle $\li$ on the variety $X$. For each $i$ we aim at cutting $X$ by divisors $H_{i,1},\dots, H_{i,n-1}$ containing $Z_i$ so as to get
	\[
		H_{i,1}\cap\cdots\cap H_{i,n-1}=Z_i\cup Z'_i,
	\]
	where $Z'_i$ is a smooth closed curve on $X$. Here for each $j$, we want the divisors $H_{i,j}$ to be defined by a section in $\ho^0(X,\li^{\otimes d_j})$ for some sufficiently large $d_j$. 
	
	On the other hand, we want these $d_j$'s to be chosen carefully so that for each $j$ we can also find a divisor $H'_{i,j}$ defined by a section in $\ho^0(X,\li^{\otimes d_j})$ so that the intersection
	\[
		Z''_i=H'_{i,1}\cap\cdots\cap H'_{i,n-1}
	\]
	is smooth of dimension $1$. Then it is known that $Z$ is rationally equivalent to $\sum_i a_i(Z''_i-Z'_i)$.\\
	
	The method of producing $Z''_i$ is a direct application of Poonen's Bertini smoothness theorem over finite fields, proved in \cite{Po04}. For each $i$, when $H'_{i,1},\dots H'_{i,j}$ are chosen so that $H'_{i,1}\cap\dots\cap H'_{i,j}$ is smooth of dimension $n-j$, we apply Poonen's theorem to $H'_{i,1}\cap\dots\cap H'_{i,j}$. Then for any $d_{j+1}\gg 1$, we can find $H'_{i,j+1}=\mathrm{div}\sigma'_{i,j+1}$ for some $\sigma'_{i,j+1}\in \ho^0(X,\li^{\otimes d_{j+1}})$ such that $H'_{i,1}\cap\dots\cap H'_{i,j}\cap H'_{i,j+1}$ is again smooth of dimension $n-j-1$.
	
	The first step needs some tricks. Intuitively, we want to proceed as in the second step. For each $i$, we want to find $H_{i,1},\cdots, H_{i,n-1}$ one by one so that for each $j$, the intersection $H_{i,1}\cap\dots\cap H_{i,j}$, containing $Z_i$, is smooth of dimension $n-j$. But this does not work well. In fact, it is not always true that we can find a smooth surface on $X$ containing a possibly singular curve $Z_i$. According to Gunther in \cite{Gu17} (already known in the infinite field case by Altman and Kleiman in \cite{KA79}), the dimension of a minimal smooth closed subscheme of $X$ containing $Z_i$ depends on the local embedding dimension of each singular point of $Z$. 
	
	Therefore we can not ask the total intersection $H_{i,1}\cap\dots\cap H_{i,j}$ to be smooth for each $j$. But instead, we can make it smooth outside $\mathrm{Sing}(Z)$, and contains $Z-\mathrm{Sing}(Z)$, until we get a surface $H_{i,1}\cap\dots\cap H_{i,n-2}$ containing $Z$ and smooth outside of $\mathrm{Sing}(Z)$. In particular, it is highly possible that $H_{i,1}\cap\dots\cap H_{i,n-2}$ is singular at the singular points of $Z$. 
	
	For the last divisor $H_{i,n-1}$, in order to avoid the possible problems caused by the singularity of $Z$, we simply add a supplementary condition that $H_{i,n-1}$ does not intersect with $\mathrm{Sing}(Z)$. The existence of this $H_{i,n-1}$ is deduced from our generalized version of Bertini theorem, as $\mathrm{Sing}(Z)$ is a finite subscheme of $Z$. Then the final intersection $H_{i,1}\cap\dots\cap H_{i,n-1}$ is $Z_i\cup Z'_i$, where $Z'_i$ satisfies the conditions we need. \\
	
	In Section \ref{generalber}, we show the proof of Theorem \ref{generalbertini} and Corollary \ref{generalbertinisingularity}, following faithfully the strategy of Poonen in \cite{Po04}. The proof of Theorem \ref{1cycle} using Corollary \ref{generalbertinisingularity} is completed in Section \ref{smoothing}.

\subsection{Notation}
\begin{enumerate}
	\item If $S$ is a set, we denote by $\# S$ its cardinality.
	
	\item If $X$ is a scheme over $\F_q$, we denote by $\vert X\rvert$ the underlying topological space of $X$.
	
	\item If $X$ is a scheme over $\F_q$, we denote by $\mathrm{Sing}(X)$ the singular locus of $X$ with induced reduced structure.
	
	\item Let $X$ be a projective variety over $\F_q$, and $\li$ a line bundle on $X$. For any subscheme $Y$ of $X$, we write
	\[
		\ho^0(Y,\li):=\ho^0(Y,\li|_Y).
	\]
	

\end{enumerate}

\subsection*{Acknowledgement}
	I am grateful to Olivier Wittenberg for asking the question answered in this article, and for introducing \cite{Hi68} to me. I also thank François Charles for helpful discussions.

\section{Generalized Bertini theorem over finite fields}\label{generalber}

In this section, we prove Theorem \ref{generalbertini} and Corollary \ref{generalbertinisingularity}.

\subsection{Singularities of low degree}
	
	Fix $c$ such that 
	\[
		\ho^1(X, \mi\otimes \li^{\otimes d})=0,
	\]
	and the natural map
	\[
		\ho^0(X, \mi\otimes \li^{\otimes d})\otimes \ho^0(X, \li) \longrightarrow \ho^0(X, \mi\otimes \li^{\otimes (d+1)})
	\]
	is surjective for all $d\geq c$.
	\begin{lem}\label{ressurj}
		Let $S$ be a finite subscheme of $X_0$. Let 
		\[
			\phi_{S,d} : \ho^0(X, \mi\otimes \li^{\otimes d})\longrightarrow \ho^0(S, \mi\otimes \li^{\otimes d})
		\]
		be the map induced by the restriction of sheaves. Then $\phi_{S,d}$ is surjective for $d\geq c+h^0(S, \Oc_S)$.
	\end{lem}
	\begin{proof}
		As $\li$ is a very ample line bundle, $\phi_{S,d}$ is surjective for $d\gg 1$. Since the surjectivity (or non-surjectivity) of $\phi_{S,d}$ does not depend on the base field, we may enlarge $\F_q$ when necessary and assume that there exists a section $\gamma\in \ho^0(X,\li)$ satisfying $\mathrm{div}\gamma\cap S=\emptyset$. Then the map
		\[
			\ho^0(X, \mi\otimes \li^{\otimes d})\stackrel{\gamma\cdot}\longrightarrow \ho^0(X, \mi\otimes \li^{\otimes (d+1)})
		\]
		is injective. Since $\gamma$ does not vanish at any point of $S$, the map
		\[
			\ho^0(S, \mi\otimes \li^{\otimes d})\stackrel{\gamma\cdot}\longrightarrow \ho^0(S, \mi\otimes \li^{\otimes (d+1)})
		\]
		is an isomorphism for each $d\geq 1$. Hence $\gamma\cdot\mathrm{Im}\phi_{S,d}\subset \mathrm{Im}\phi_{S,d+1}$, and in particular, 
		\[
			\dim\mathrm{Im}\phi_{S,d}\leq \dim\mathrm{Im}\phi_{S,d+1}.
		\]
		If for some $d\geq c$ we have the equality of the dimensions, then 
		\[
			\gamma\cdot \mathrm{Im}\phi_{S,d} = \mathrm{Im}\phi_{S,d+1}.
		\]
		For any $d\geq c$, we have the following commutative diagram :
		\[
			\xymatrix{
				\ho^0(X,\li)\otimes\ho^0(X, \mi\otimes \li^{\otimes d}) \ar@{->>}[r] \ar[d] & \ho^0(X, \mi\otimes \li^{\otimes (d+1)})\  \ar[d] \\
				\ho^0(X,\li)\otimes\ho^0(S, \mi\otimes \li^{\otimes d}) \ar[r]& \ho^0(S, \mi\otimes \li^{\otimes (d+1)}).
			}
		\]
		In particular, as the top line is surjective, the map
		\[
			\ho^0(X,\li)\otimes \mathrm{Im}\phi_{S,d}\longrightarrow \mathrm{Im}\phi_{S,(d+1)}
		\]
		is also surjective. If $\gamma\cdot \mathrm{Im}\phi_{S,d} = \mathrm{Im}\phi_{S,d+1}$ for some $d\geq c$, then
		\begin{eqnarray*}
			\mathrm{Im}\phi_{S,d+2} &=& \ho^0(X,\li)\cdot \mathrm{Im}\phi_{S,d+1} \\
			&=& \ho^0(X,\li)\cdot \gamma\cdot \mathrm{Im}\phi_{S,d} \\
			&=& \gamma\cdot\left( \ho^0(X,\li)\cdot \mathrm{Im}\phi_{S,d} \right) \\
			&=& \gamma\cdot \mathrm{Im}\phi_{S,d+1}.
		\end{eqnarray*}
		Therefore
		\[
			\dim\mathrm{Im}\phi_{S,d+2} = \dim\mathrm{Im}\phi_{S,d+1}.
		\]
		So for any $d'>d$,
		\[
			\dim\mathrm{Im}\phi_{S,d'} = \dim\mathrm{Im}\phi_{S,d},
		\]
		and we must have
		\[
			\mathrm{Im}\phi_{S,d}=\ho^0(S,\mi\otimes \li^{\otimes d}).
		\]
		As $\dim\mathrm{Im}\phi_{S,c}\geq 0$ and $\ho^0(S,\mi\otimes \li^{\otimes d})\simeq \ho^0(S,\Oc_S)$ (non-canonically), starting from $c$, each time $d$ grows by $1$, the dimension of the image of $\phi_{S,d}$ increases by at least $1$ until we get the surjectivity of $\phi_{S,d}$. Therefore we must have that $\phi_(S,d)$ is surjective for $d\geq c+h^0(S, \Oc_S)$.
	\end{proof}
	\begin{lem}[Singularities of small degree]\label{small}
		Under the same hypotheses as in Theorem \ref{generalbertini}, define
		\[
			\mathcal{P}_{\leq r,d}=\left\{ \sigma\in \ho^0(X,\mi\otimes \li^{\otimes d})\ ;\ 
			\begin{array}{lll}
			\forall x\in \lvert \mathrm{div}\sigma\cap X_0 \rvert,\ \deg x\leq r,\\
			 \mathrm{div}\sigma \text{ is smooth of dim } n-1 \text{ at } x 
			\end{array}\right\},
		\]
		and $\mathcal{P}_{\leq r}=\bigcup_{d\geq 0}\mathcal{P}_{\leq r,d}$. Then there exists $d_1\in \Z_{>0}$, such that for any $d\geq d_1$, we have
		\begin{eqnarray*}
			\frac{\#\mathcal{P}_{\leq r,d}}{\#\ho^0(X,\mi\otimes \li^{\otimes d})}
			= \left(\prod_{x\in\lvert X_0\rvert,\ \deg x\leq r} (1-q^{-(1+n)\deg x}) \right).
		\end{eqnarray*}
		In particular, we have
		\[
			\mu(\mathcal{P}_{\leq r})=\prod_{x\in\lvert X_0\rvert,\ \deg x\leq r} (1-q^{-(1+n)\deg x}),
		\]
		and hence
		\[
			\lim_{r\rightarrow\infty}\mu(\mathcal{P}_{\leq r})=\zeta_{X_0}(1+n)^{-1}.
		\]
	\end{lem}
	
	\begin{proof}
		For any $x\in \lvert X \rvert$, we denote by $x'$ the first order infinitesimal neighbourhood of $x$ in $X$. If $x\in \lvert X_0\rvert $, then for any line bundle $\mathcal{M}$ on $X$, we have
		\[
			\#\ho^0(x', \mathcal{M})=\#\ho^0(x', \Oc_{x'})=q^{(1+n)\deg x}.
		\]
		For such a closed point $x$, $\mathrm{div}\sigma$ is not smooth of dimension $n-1$ at $x$ if and only if the restriction of $\sigma\in \ho^0(X, \mi\otimes \li^{\otimes d})$ to $\ho^0(x', \mi\otimes \li^{\otimes d})$ is $0$.

		Now let $X'_{\leq r}$ be the finite subscheme
		\[
			X'_{\leq r}=\coprod_{x\in\lvert X_0 \rvert,\ \deg x\leq r}x'.
		\]
		Then for any $d\geq 0$,
		\begin{eqnarray*}
			\ho^0(X'_{\leq r}, \mi\otimes \li^{\otimes d})
			\simeq \prod_{x\in\lvert X_0 \rvert,\ \deg x\leq r}\ho^0(x', \mi\otimes \li^{\otimes d}) ,
		\end{eqnarray*}
		and $\sigma\in \mathcal{P}_{\leq r,d}$ if and only if when we restrict $\sigma$ to $\ho^0(X'_{\leq r}, \mi\otimes \li^{\otimes d})$, the image is non-zero in each factor of the product. 
		
		As $\li$ is very ample, we can find $d_1>0$ such that the restriction
		\[
			\ho^0(X, \mi\otimes \li^{\otimes d})\longrightarrow \ho^0(X'_{\leq r}, \mi\otimes \li^{\otimes d})
		\]
		is surjective for any $d\geq d_1$. For such $d$, we have
		\begin{eqnarray*}
			\frac{\#\mathcal{P}_{\leq r,d}}{\#\ho^0(X,\mi\otimes \li^{\otimes d})} = \prod_{x\in\lvert X_0\rvert,\ \deg x\leq r} \frac{q^{(1+n)\deg x}-1}{q^{(1+n)\deg x}}.
		\end{eqnarray*}
		This turns out to be the result of the lemma.
	\end{proof}
	
	To prove the variant of the theorem, we need another lemma.

	\begin{lem}
		Under the same hypotheses as in Theorem \ref{generalbertini} and assuming that $S$ is a finite subscheme of $Y$ as in Corollary \ref{generalbertinisingularity}, define
		\[
			\mathcal{P}_{S,\leq r,d}=\left\{ \sigma\in \ho^0(X,\mi\otimes \li^{\otimes d})\ ;\ 
			\begin{array}{lll}
			\forall x\in \lvert \mathrm{div}\sigma\cap X_0 \rvert,\ \deg x\leq r,\\
			 \mathrm{div}\sigma \text{ is smooth of dim } n-1 \text{ at } x, \\
			 \mathrm{div}\sigma\cap S=\emptyset
			\end{array}\right\},
		\]
		and $\mathcal{P}_{S,\leq r}=\bigcup_{d\geq 0}\mathcal{P}_{S,\leq r, d}$. Then we have
		\[
			\mu(\mathcal{P}_{S,\leq r})=\prod_{x\in\lvert X_0\rvert,\ \deg x\leq r} (1-q^{-(1+n)\deg x})\cdot\prod_{x\in\lvert S\rvert} \left(1-q^{-\deg x}\right),
		\]
		and hence
		\[
			\lim_{r\rightarrow\infty}\mu(\mathcal{P}_{S,\leq r})=\zeta_{X_0}(1+n)^{-1}\prod_{x\in\lvert S\rvert} \left(1-q^{-\deg x}\right).
		\]
	\end{lem}
	\begin{proof}
		We proceed in a similarly way to the proof of Lemma \ref{small}, but consider the finite scheme $X'_{\leq r}\cup \left(\coprod_{x\in \lvert S\rvert }x\right)$ instead of $X'_{\leq r}$. Then we have
		\begin{eqnarray*}
			\ho^0(X'_{\leq r}\cup \coprod_{x\in \lvert S\rvert }x, \mi\otimes \li^{\otimes d})
			\simeq \left(\prod_{x\in\lvert X_0 \rvert,\ \deg x\leq r}\ho^0(x', \mi\otimes \li^{\otimes d})\right)\times \left(\prod_{x\in\lvert S\rvert}\ho^0(x, \mi\otimes \li^{\otimes d})\right),
		\end{eqnarray*}
		and a section $\sigma$ is contained in $\mathcal{P}_{S,\leq r,d}$ if and only if when we restrict $\sigma$ to $\ho^0(X'_{\leq r}\cup \coprod_{x\in \lvert S\rvert }x, \mi\otimes \li^{\otimes d})$, the image is non-zero in each factor on the right side. In fact, $\mathrm{div}\sigma\cap S=\emptyset$ if and only if $\sigma(x)\not=0$ in $\ho^0(x, \mi\otimes \li^{\otimes d})$ for each $x\in \lvert S \rvert$.
		
		As $\li$ is very ample, we can find $d'_1>0$ such that the restriction
		\[
			\ho^0(X, \mi\otimes \li^{\otimes d})\longrightarrow \ho^0(X'_{\leq r}\cup \coprod_{x\in \lvert S\rvert }x, \mi\otimes \li^{\otimes d})
		\]
		is surjective for any $d\geq d'_1$. For such $d$, we have
		\begin{eqnarray*}
			\frac{\#\mathcal{P}_{S,\leq r,d}}{\#\ho^0(X,\mi\otimes \li^{\otimes d})} = \prod_{x\in\lvert X_0\rvert,\ \deg x\leq r} \frac{q^{(1+n)\deg x}-1}{q^{(1+n)\deg x}}
			\cdot \prod_{x\in \lvert S\rvert}\frac{q^{\deg x}-1}{q^{\deg x}}.
		\end{eqnarray*}
		This proportion gives us the result for $\mu(\mathcal{P}_{S,\leq r})$ and $\lim_{r\rightarrow\infty}\mu(\mathcal{P}_{S,\leq r})$.
	\end{proof}
	
\subsection{Singularities of medium degree}
	
	\begin{lem}\label{closurj}
		Under the same hypotheses as in Theorem \ref{generalbertini}, let $x$ be a closed point in $X_0$ of degree $e$, where $e\leq \frac{d-c}{n+1}$. Then the proportion of $\sigma \in \ho^0(X,\mi\otimes \li^{\otimes d})$ such that $\mathrm{div}\sigma$ is singular at $x$ equals $q^{-(1+n)e}$.
	\end{lem}
	\begin{proof}
		The lemma follows by applying Lemma \ref{ressurj} to $S=x'$, where $x'$ is the first order infinitesimal neighbourhood of $x$ in $X$.
	\end{proof}
	
	\begin{lem}[Singularities of medium degree]\label{med}
		Under the same hypotheses as in Theorem \ref{generalbertini}, define
		\[
			\mathcal{Q}_{r,d}^{\mathrm{med}}:=\left\{ \sigma\in \ho^0(X,\mi\otimes \li^{\otimes d})\ ;\ 
			\begin{array}{ll}
				\exists x\in X_0,\ r<\deg x\leq\frac{d-c}{n+1}, \\
				\mathrm{div}\sigma \text{ is singular at }x
			\end{array}
			\right\}.
		\]
		Set $\mathcal{Q}_{r}^{\mathrm{med}}=\bigcup_{d\geq 0}\mathcal{Q}_{r,d}^{\mathrm{med}}$. Then 
		\[
			\lim_{r\rightarrow \infty}\overline{\mu}(\mathcal{Q}_{r}^{\mathrm{med}})=0.
		\]
	\end{lem}
	\begin{proof}
		By the Lang-Weil estimates in \cite{LW54}, there exists a constant $c_0>0$ such that for any $e\in\Z_{>0}$, 
		\[
			\# X(\F_{q^e})\leq c_0q^{ne}.
		\]
		Then by Lemma \ref{closurj}, we have
		\begin{eqnarray*}
			\frac{\#\mathcal{Q}_{r,d}^{\mathrm{med}}}{\#\ho^0(X,\mi\otimes \li^{\otimes d})} 
			&\leq&\sum_{x\in \lvert X_0\rvert,\ r< \deg x\leq\frac{d-c}{1+n}}q^{-(1+n)e} \\
			&\leq& \sum_{e=r+1}^{+\infty}c_0q^{ne}q^{-(1+n)e}\\
			&=& 2c_0\sum_{e=r+1}^{+\infty}q^{-e}\leq 4c_0q^{-r}.
		\end{eqnarray*}
		When $r$ tends to infinity, $4c_0q^{-r}\rightarrow 0$ and hence $\lim_{r\rightarrow \infty}\overline{\mu}(\mathcal{Q}_{r}^{\mathrm{med}})=0$.
	\end{proof}
	
\subsection{Singularities of high degree}

	\begin{lem}[Singularities of high degree]\label{high}
		Under the same hypotheses as in Theorem \ref{generalbertini}, define
		\[
			\mathcal{Q}_{d}^{\mathrm{high}}:=\left\{ \sigma\in \ho^0(X,\mi\otimes \li^{\otimes d})\ ;\ 
			\begin{array}{ll}
				\exists x\in X_0,\ \deg x>\frac{d-c}{1+n} \text{ s.t. }  \\
				\mathrm{div}\sigma \text{ is singular at }x
			\end{array}
			\right\}.
		\]
		Set $\mathcal{Q}^{\mathrm{high}}=\bigcup_{d\geq 0}\mathcal{Q}_{d}^{\mathrm{high}}$. Then $\overline{\mu}(\mathcal{Q}^{\mathrm{high}})=0$.
	\end{lem}

	For an open subset $U$ of $X_0$, we define
	\[
		\mathcal{Q}_{U, d}^{\mathrm{high}}:=\left\{ \sigma\in \ho^0(X,\mi\otimes \li^{\otimes d})\ ;\ 
		\begin{array}{ll}
			\exists x\in U,\ \deg x>\frac{d-c}{1+n} \text{ s.t. }  \\
			\mathrm{div}\sigma \text{ is singular at }x
		\end{array}
		\right\}
	\]
	and $\mathcal{Q}_{U}^{\mathrm{high}}=\bigcup_{d\geq 0}\mathcal{Q}_{U, d}^{\mathrm{high}}$. Then for each finite open cover $\{U_i\}_{i\in I}$ of $X_0$, the lemma follows if $\overline{\mu}(\mathcal{Q}_{U_i}^{\mathrm{high}})=0$ for each $U_i$ of the open cover. 
	
	Since $X_0$ is smooth of dimension $n$, the sheaf $\Omega^1_{X_0/\F_q}$ is locally free of rank $n$. We may cover $X_0$ by finitely many open subset $U$ such that we can find a trivalization $\psi_{U}: \mi|_U\longrightarrow \Oc_U$ and that there exist local sections $t_1,\dots, t_n\in \ho^0(U,\Oc_{X})$ with
	\[
		\Omega_{U/\F_q}^1\simeq \bigoplus_{i=1}^n\Oc_U \mathrm{d} t_i.
	\]
	Note that the trivialization $\psi_{U}$ induces isomorphisms $(\mi\otimes \li^{\otimes d})|_U\longrightarrow (\li^{\otimes d})|_U$ for each $d\geq 0$, which we will still denote by $\psi_U$. Moreover, choosing a finer finite open cover of $X_0$ when needed, we may assume that for some $N_1>0$ not divisible by $p$, there exists a section $\tau\in \ho^0(X, \li^{\otimes N_1})$ such that  $U=X-\mathrm{div}\tau$. 
	
	So to prove Lemma \ref{high}, it is enough to show that $\overline{\mu}(\mathcal{Q}_U^{\mathrm{high}})=0$ for such open subsets $U$ of $X_0$.\\

	For each $i=1,\dots, n$, let $\partial_i$ be the dual of $\mathrm{d} t_i$ in 
	\[
		Der_{\F_q}(\mathcal{O}_U, \mathcal{O}_U)\simeq \Hom_{\mathcal{O}_U}(\Omega^1_{U/\F_q}, \mathcal{O}_U)\simeq \ho^0\left( U, \mathcal{H}om(\Omega^1_{X/\F_q}, \mathcal{O}_X) \right).
	\]
	When we have a section $\sigma\in \ho(X, \mi\otimes \li^{\otimes d})$, whether $\mathrm{div}\sigma$ is singular at a closed point $x$ of $U$ can be tested via $\psi_U(\sigma|_U)$. In fact, as 
	\[
		\mathrm{div}\sigma\cap U=\mathrm{div}\psi_U(\sigma|_U),
	\]
	$\mathrm{div}\sigma$ is singular at $x$ if and only if 
	\[
		\psi_U(\sigma|_U)(x)=\partial_1\left(\psi_U(\sigma|_U)\right)(x)=\cdots=\partial_n\left(\psi_U(\sigma|_U)\right)(x)=0,
	\]
	i.e.
	\[
		\mathrm{Sing}(\mathrm{div}\sigma)\cap U=\mathrm{div}\psi_U(\sigma|_U)\cap\bigcap_{i=1}^n \mathrm{div}\ \partial_i\left(\psi_U(\sigma|_U)\right).
	\]
	But it is not easy to bound the intersection on the right side directly. To fix this problem, we need to extend the local sections $\partial_i\left(\psi_U(\sigma|_U)\right)$ to sections on $X$. While this may not be possible directly, the following two lemmas tells us that we can always extend these sections by multiplying a certain power of the section $\tau$.
	\begin{lem}
		There exists an integer $d_2>0$ such that for any $d\geq 0$, $\tau^{d_2}\psi_U$ determines a map on $X$
		\[
			\ho^0(X, \mi\otimes \li^{\otimes d})\longrightarrow \ho^0(X, \li^{\otimes (d+N_1d_2)})
		\]
		sending $\sigma$ to $\tau^{d_2}\psi_U(\sigma|_U)$, where the isomorphism $\psi_U :\mi|_U\longrightarrow \Oc_U$ and the section $\tau\in \ho^0(X, \li^{\otimes N_1})$ is defined as above.
	\end{lem}
	\begin{proof}
		The trivialization $\psi_U$ can be regarded as a section in $\ho^0(U, \mathcal{H}om(\mi, \Oc_X))$. As $\li$ is ample, we can find an integer $d_2>0$ such that $\psi_U\cdot \tau^{d_2}$ extends to a global section in 
		\[
			\ho^0(X, \mathcal{H}om(\mi, \Oc_X)\otimes \li^{N_1d_2})\simeq \mathrm{Hom}(\mi, \li^{N_1d_2}).
		\]
		We fix an extension implicitly. Then for any $d\geq 0$, $\psi_U\cdot \tau^{d_2}$ induces a morphism 
		\[
			\mi\otimes \li^{\otimes d}\longrightarrow \li^{\otimes (d+N_1d_2)},
		\]
		and hence a morphism of global sections 
		\[
			\ho^0(X, \mi\otimes \li^{\otimes d})\longrightarrow \ho^0(X, \li^{\otimes (d+N_1d_2)})
		\]
		sending $\sigma$ to $\tau^{d_2}\psi_U(\sigma|_U)$. 
	\end{proof}
		
	Similarly, we have :
	\begin{lem}
		There exists an integer $d_3>0$ such that for each $1\leq i\leq n$ and any $d\geq 0$, $\tau^{d_3}\cdot \partial_i$ determines a map on $X$
		\[
			\ho^0(X,\li^{\otimes d})\longrightarrow \ho^0(X,\li^{\otimes (d+N_1d_3)})
		\]
		sending $\sigma$ to $\tau^{d_3}\partial_i(\sigma|_U)$.
	\end{lem}	
	\begin{proof}
		In fact, we can find $d_3>0$ such that each $\tau^{d_3}\cdot \partial_i$ can be extended to an element in 
		\[
			\ho^0\left( X, \mathcal{H}om(\Omega^1_{X/\F_q}, \mathcal{O}_X)\otimes \li^{N_1d_3} \right).
		\]
		Then the extended element induces the map in the statement. 
	\end{proof}
	\begin{rmq}
		The two result of extensions are not at all canonical. We can find a lot of extension maps satisfying the same conditions. But for our purpose, we always assume that we fix a choice of these extension maps implicitly.
	\end{rmq}
	
		Now for each $d>0$ and each $i$, we have the following composition of maps
		\[
			\ho^0(X,\mi\otimes \li^{\otimes d}) \longrightarrow \ho^0(X,\li^{\otimes (d+N_1d_2)}) \longrightarrow \ho^0(X,\li^{\otimes (d+N_1(d_2+d_3))})
		\]
		sending a section $\sigma\in \ho^0(X,\mi\otimes \li^{\otimes d})$ to $\tau^{d_2}\psi_U(\sigma|_U)\in\ho^0(X,\li^{\otimes (d+N_1d_2)})$ and then to 
		\[
			\tau^{d_3}\cdot \partial_i\left(\Big(\tau^{d_2}\psi_U(\sigma|_U)\Big)|_U\right)\in \ho^0(X,\li^{\otimes (d+N_1(d_2+d_3))}).
		\]
		We denote by $D_i$ this composition. So the above composition can be written as 
		\[
			D_i\ :\ \ho^0(X,\mi\otimes \li^{\otimes d}) \longrightarrow \ho^0(X,\li^{\otimes (d+N_1(d_2+d_3))})
		\]
		sending $\sigma$ to $D_i(\sigma)=\tau^{d_3}\cdot \partial_i\left(\Big(\tau^{d_2}\psi_U(\sigma|_U)\Big)|_U\right)$.

		 Note that $\psi_U$ is an isomorphism on $U$, and that $\tau$ does not vanish on any point of $U$. For a section $\sigma\in \ho^0(X,\mi\otimes \li^{\otimes d})$, if $\mathrm{div}\sigma$ is singular at a closed point $x\in \lvert U\rvert$, then so is $\mathrm{div}(\tau^{d_2}\psi_U(\sigma|_U))$. The last condition is then equivalent to 
		\[
			\tau^{d_2}\psi_U(\sigma|_U)(x)=0 \qquad \text{ and }\quad \partial_i\left(\Big(\tau^{d_2}\psi_U(\sigma|_U)\Big)|_U\right)(x)=0,\ i=1,\dots,n,
		\]
		which are equivalent to
		\[
			\sigma(x)=0 \qquad \text{ and }\quad D_1(\sigma)(x)=0,\ i=1,\dots,n.
		\]
		Therefore
		\[
			\mathrm{Sing}(\mathrm{div}\sigma)\cap U\subset \mathrm{Sing}\big(\mathrm{div}(\tau^{d_2}\psi_U(\sigma|_U))\big)=\mathrm{div}\sigma \cap \bigcap_{i=1}^n\mathrm{div}D_1(\sigma)\cap U.
		\]
		
		We may also choose $d'_4>0$ so that $\tau^{d'_4}\cdot \psi_U^{-1}$, and also all the sections $\tau^{d'_4}\cdot t_1\psi_U^{-1},\dots, \tau^{d'_4}\cdot t_n\psi_U^{-1}$ on $U$, extend to global morphisms $\li^{\otimes d}\longrightarrow \mi\otimes \li^{\otimes (d+N_1d'_4)}$ for each $d\geq 0$. Here $\psi_U^{-1}$ is the inverse of the trivialization map $\psi_U$. If $d'_4$ satisfies this condition, then so does any integer larger than $d'_4$. Now for $d\gg 1$, we choose $d_4$ to be the minimal integer satisfying $d_4\geq d'_4$ and $p | d-N_1d_4$. In particular, $d_4$ satisfies the condition for $d'_4$ and
		\[
			d'_4\leq d_4 <d'_4+p.
		\]
		Let $k_d=(d-N_1d_4)/p$. If we choose $\sigma_0\in \ho^0(X,\mi\otimes \li^{\otimes d})$, $\beta_1,\dots, \beta_n, \gamma\in \ho^0(X,\li^{\otimes k_d})$ uniformly at random, then 
		\[
			\sigma=\sigma_0+\sum_{i=1}^n\tau^{d_4}t_i\psi_U^{-1}(\beta_i^p)+\tau^{d_4}\psi_U^{-1}(\gamma^p)
		\]
		is a random element in $\ho^0(X,\mi\otimes \li^{\otimes d})$. For each $1\leq j\leq n$, we have
		\[
			D_j(\sigma)=D_j(\sigma_0)+\sum_{i=1}^nD_j\left(\tau^{d_4}t_i\psi_U^{-1}(\beta_i^p)\right)+D_j\left(\tau^{d_4}\psi_U^{-1}(\gamma^p)\right).
		\]
		When restricted to $U$, we have
		\begin{eqnarray*}
			D_j\left(\tau^{d_4}t_i\psi_U^{-1}(\beta_i^p)\right) &=& \tau^{d_3}\cdot \partial_j\left(\Big(\tau^{d_2}\psi_U(\tau^{d_4}t_i\psi_U^{-1}(\beta_i^p))\Big)|_U\right) \\
			&=& \tau^{d_3}\cdot \partial_j\left(\tau^{d_2+d_4}t_i\beta_i^p\right) \\
			&=& \tau^{d_3}t_i\beta_i^p\partial_j(\tau^{d_2+d_4})+\tau^{d_2+d_3+d_4}\beta_i^p\partial_j(t_i)+0 \\
			&=& \left\{\begin{array}{ll}
				\tau^{d_3}t_j\beta_j^p\partial_j(\tau^{d_2+d_4})+\tau^{d_2+d_3+d_4}\beta_j^p,\qquad i=j \\
				\tau^{d_3}t_i\beta_i^p\partial_j(\tau^{d_2+d_4}),\qquad\qquad\qquad\qquad\ \   i\not=j
				\end{array} \right. ,
		\end{eqnarray*}
		and 
		\begin{eqnarray*}
			D_j\left(\tau^{d_4}\psi_U^{-1}(\gamma^p)\right) &=&\tau^{d_3}\cdot \partial_j\left(\Big(\tau^{d_2}\psi_U(\tau^{d_4}\psi_U^{-1}(\gamma^p))\Big)|_U\right)\\
			&=&\tau^{d_3}\cdot \partial_j\left(\tau^{d_2+d_4}\gamma^p\right) \\
			&=&\tau^{d_3}\gamma^p\partial_j(\tau^{d_2+d_4}).
		\end{eqnarray*}
		Hence on $U$, we can write
		\begin{eqnarray*}
			D_j(\sigma) &=& D_j(\sigma_0)+\left(\sum_{i=1}^nt_i\beta_i^p+\gamma^p\right)\tau^{d_3}\partial_j(\tau^{d_2+d_4})+\tau^{d_2+d_3+d_4}\beta_j^p \\
			&=&D_j(\sigma_0)+(d_2+d_4)\left(\sum_{i=1}^nt_i\beta_i^p+\gamma^p\right)\tau^{d_2+d_3+d_4-1}\partial_j(\tau)+\tau^{d_2+d_3+d_4}\beta_j^p \\
			&=& D_j(\sigma_0)+(d_2+d_4)\Big(\psi_U(\sigma)-\psi_U(\sigma_0)\Big)\tau^{d_2+d_3-1}\partial_j(\tau)+\tau^{d_2+d_3+d_4}\beta_j^p.
		\end{eqnarray*}
		$ $\\
		
		For any element
		\[
			\big( \sigma_0, (\beta_1,\dots, \beta_n),\gamma \big) \in \ho^0(X,\mi\otimes \li^{d})\times \left(\prod_{i=1}^n\ho^0(X,\li^{\otimes k_d})\right)\times \ho^0(X,\li^{\otimes k_d}),
		\]
		we define
		\[
			g_{j}(\sigma_0,\beta_j)=D_j(\sigma_0)-(d_2+d_4)\psi_U(\sigma_0)\tau^{d_2+d_3-1}\partial_j(\tau)+\tau^{d_2+d_3+d_4}\beta_j^p \in \ho^0(U,\li^{\otimes (d+N_1(d_2+d_3))}),
		\]
		and
		\[
			W_{j}:=U\cap \{g_{1}=\cdots=g_{j}=0\}.
		\]
		Then for $\sigma=\sigma_0+\sum_{i=1}^n\tau^{d_4}t_i\psi_U^{-1}(\beta_i^p)+\tau^{d_4}\psi_U^{-1}(\gamma^p)$, comparing the expression of $g_{j}$ and $D_j(\sigma)$, which gives
		\[
			g_{j}(\sigma_0,\beta_j)=D_j(\sigma)-\psi_U(\sigma)(d_2+d_4)\tau^{d_2+d_3-1}\partial_j(\tau),
		\]
		we get
		\[
			g_{j}(\sigma_0,\beta_j)|_{U\cap\mathrm{div}\sigma}=D_j(\sigma)|_{U\cap\mathrm{div}\sigma}.
		\]
		Hence
		\[
			W_{j}\cap \mathrm{div}\sigma=\mathrm{div}\sigma\cap \mathrm{div}D_1(\sigma)\cap\cdots\cap \mathrm{div}D_j(\sigma)\cap U.
		\]
		In particular,
		\[
			\mathrm{Sing}(\mathrm{div}\sigma)\cap U\subset W_{n}\cap \mathrm{div}\sigma.
		\]
		
		As $\tau^{d_3}\cdot \partial_i$ extends to $X$, so does the section $\tau^{d_2+d_3-1}\partial_j(\tau)$. Therefore $g_j(\sigma_0,\beta_j)$ extends to an element in $\ho^0(X,\li^{\otimes (d+N_1(d_2+d_3))})$.
		
	\begin{lem}\label{need1}
		For $0\leq j\leq n-1$, with a fixed choice of $\sigma_0,\beta_1,\dots, \beta_j$ such that $\dim W_j\leq n-j$, the proportion of $\beta_{j+1}$ in $\ho^0(X,\li^{\otimes k_d})$ such that $\dim W_{j+1}\leq n-j-1$ is $1-o(1)$ as $d\rightarrow \infty$. Here $o(1)$ depends on $U$ and the $D_j$'s.
	\end{lem}
	\begin{proof}
		Applying refined Bézout's theorem \cite[p.10]{Fu84}, the number of irreducible $(n-j)$-dimensional $\F_q$-components of $W_j$ is bounded by
		\[
			(\deg_{\li}X)(\deg g_{1})\cdots (\deg g_{j})=\deg_{\li}X\cdot  \left(d+N_1(d_2+d_3) \right)^j=O(d^j),
		\]
		when $d\rightarrow \infty$.
		
		On each irreducible $n-j$-dimensional $\F_q$-components $W_{j,e}$ of $W_j$, we need to bound the number of $\beta_{j+1}$ that makes $g_{j+1}(\sigma_0,\beta_{j+1})$ vanishing identically on $W_{j,e}$. We denote the subset of such $\beta_{j+1}$ in $\ho^0(X,\li^{\otimes k_d})$ by $G_{j,e}^{bad}$. Note that if $\beta_{j+1}$ and $\beta'_{j+1}$ are both contained in $G_{j,e}^{bad}$, then 
		\[
			g_{j+1}(\sigma_0,\beta_{j+1})-g_{j+1}(\sigma_0,\beta_{j+1})=\tau^{d_2+d_3+d_4}(\beta_{j+1}-\beta'_{j+1})^p
		\]
		vanishes identically on $W_{j,e}$. Therefore $\beta_{j+1}-\beta'_{j+1}$ is identically $0$ when restricted to $W_{j,e}$. So $G_{j,e}^{bad}$, if non-empty, is a coset of the subspace of sections in $\ho^0(X,\li^{\otimes k_d})$ vanishing on $W_{j,e}$. 
		
		Since $\li$ is very ample on $X$, we may find at least one section $\tau_e\in \ho^0(X,\li)$ that does not vanish identically on $W_{j,e}$. Then sections 
		\[
			\tau_e^{k_d}, \tau\tau_e^{k_d-N_1},\dots, \tau^{\lfloor\frac{k_d}{N_1}\rfloor}\tau_e^{k_d-\lfloor\frac{k_d}{N_1}\rfloor}\in \ho^0(X,\li^{\otimes k_d})
		\]
		have linearly independent images when restricting to $W_{j,e}$. Hence the subspace of sections in $\ho^0(X,\li^{\otimes k_d})$ vanishing on $W_{j,e}$ has codimension at least $\lfloor\frac{k_d}{N_1}\rfloor+1\geq \frac{k_d}{N_1}$. Therefore, as $k_d=(d-N_1d_4)/p$, we have
		\begin{eqnarray*}
			\frac{\#G_{j,e}^{bad}}{\#\ho^0(X,\li^{\otimes k_d})} &\leq& q^{-\frac{k_d}{N_1}}=q^{-\frac{d-N_1d_4}{pN_1}}.
		\end{eqnarray*}
		When $d\gg 1$, the proportion of $\beta_{j+1}$ such that $g_{j+1}(\sigma_0,\beta_{j+1})$ vanishes identically on at least one of the irreducible $n-j$-dimensional $\F_q$-components of $W_j$ is bounded by 
		\[
			O(d^j)\cdot q^{-\frac{d-N_1d_4}{pN_1}}=O(d^j q^{-\frac{d}{2pN_1}})=o(1).
		\]
		This proves the lemma.
	\end{proof}
	\begin{lem}\label{need2}
		With a fixed choice of $\sigma_0,\beta_1,\dots, \beta_n$ such that $W_n$ is finite, when $d\gg 1$, the proportion of $\gamma\in \ho^0(X,\li^{\otimes k_d})$ such that 
		\[
			\mathrm{div}\sigma\cap W_n\cap \left\{ x\in \lvert U\rvert \ ;\ \deg x>\frac{d-c}{1+n} \right\}=\emptyset
		\]
		is $1-o(1)$, where $\sigma=\sigma_0+\sum_{i=1}^n\tau^{d_4}t_i\psi_U^{-1}(\beta_i^p)+\tau^{d_4}\psi_U^{-1}(\gamma^p)$.
	\end{lem}
	\begin{proof}
		Applying refined Bézout's theorem as in the proof of the above lemma, we know that $\#W_n=O(d^n)$. For a given $x\in\lvert W_n \rvert$, the subset $H_x^{bad}$ of $\gamma\in \ho^0(X,\li^{\otimes k_d})$ making $\sigma(x)=0$ is either empty or a coset of 
		\[
			\Ker \left( \ho^0(X,\li^{\otimes k_d})\longrightarrow \ho^0(x,\li^{\otimes k_d}) \right).
		\]
		Applying Lemma 2.5 of \cite{Po04}, where we need to consider $\ho^0(X,\li^{\otimes k_d})$ as $\ho^0\left(\pr\big(\ho^0(X,\li)^{\vee} \big), \Oc(k_d)\right)$ and then as degree $\leq k_d$ polynomials on an open affine subspace of $\pr\big(\ho^0(X,\li)^{\vee} \big)$ containing $x$, we get
		\[
			\frac{\#\Ker \left( \ho^0(X,\li^{\otimes k_d})\longrightarrow \ho^0(x,\li^{\otimes k_d}) \right)}{\#\ho^0(X,\li^{\otimes k_d})} \leq q^{-\min(k_d+1, \frac{d-c}{1+n})}.
		\]
		Hence 
		\[
			\frac{\#\left( \bigcup_{x\in\lvert W_n \rvert }H_x^{bad} \right)}{\#\ho^0(X,\li^{\otimes k_d})}\leq O(d^n)q^{-\min(k_d+1, \frac{d-c}{1+n})}=o(1),
		\]
		and the lemma follows.
	\end{proof}
	\begin{proof}[Proof of Lemma \ref{high}]
		We only need to show that $\overline{\mu}(\mathcal{Q}_U^{\mathrm{high}})=0$. If we choose 
		\[
			\big( \sigma_0, (\beta_1,\dots, \beta_n),\gamma \big) \in \ho^0(X,\mi\otimes \li^{d})\times \left(\prod_{i=1}^n\ho^0(X,\li^{\otimes k_d})\right)\times \ho^0(X,\li^{\otimes k_d}),
		\]
		randomly, then $\sigma=\sigma_0+\sum_{i=1}^n\tau^{d_4}t_i\psi_U^{-1}(\beta_i^p)+\tau^{d_4}\psi_U^{-1}(\gamma^p)$ is a random element in $\ho^0(X,\mi\otimes \li^{d})$. Lemma \ref{need1} and \ref{need2} show that when $d\rightarrow \infty$, the proportion of $\big( \sigma_0, (\beta_1,\dots, \beta_n),\gamma \big)$ such that $\dim W_j\leq n-j$ for each $1\leq j\leq n$ and 
		\[
			\mathrm{div}\sigma\cap W_n\cap \left\{ x\in \lvert U\rvert \ ;\ \deg x>\frac{d-c}{1+n} \right\}=\emptyset
		\]
		is
		\[
			\left(\prod_{j=1}^n (1-o(1))\right)\cdot (1-o(1))=1-o(1).
		\]
		As 
		\[
			\mathrm{Sing}(\mathrm{div}\sigma)\cap U\subset W_{n}\cap \mathrm{div}\sigma,
		\]
		we get that
		\[
			\frac{\#\mathcal{Q}_{U,d}^{\mathrm{high}}}{\#\ho^0(X,\mi\otimes \li^{d})}=1-(1-o(1))=o(1).
		\]
		Therefore $\overline{\mu}(\mathcal{Q}_U^{\mathrm{high}})=0$, and hence $\overline{\mu}(\mathcal{Q}^{\mathrm{high}})=0$, which is the result of Lemma \ref{high}.
	\end{proof}

\subsection{Conclusion}	
	
	\begin{proof}[Proof of Theorem \ref{generalbertini}]
		Note that for any $d\geq 0$, we have
		\[
			\mathcal{P}_d\subset \mathcal{P}_{\leq r, d}\subset \mathcal{P}_d\cup \mathcal{Q}_{r,d}^{\mathrm{med}}\cup \mathcal{Q}_{d}^{\mathrm{high}}.
		\]
		Hence
		\[
			\left\lvert \#\mathcal{P}_d-\mathcal{P}_{\leq r, d}\right\rvert\leq \#\mathcal{Q}_{r,d}^{\mathrm{med}}+\mathcal{Q}_{d}^{\mathrm{high}}.
		\]
		Dividing the inequality by $\#\ho^0(X,\mi\otimes \li^{d})$ and taking the limit for $d\rightarrow \infty$, we get that $\overline{\mu}(\mathcal{P})$ and $\underline{\mu}(\mathcal{P})$ differ from $\mu(\mathcal{P}_{\leq r})$ by at most $\overline{\mu}(\mathcal{Q}_{r,d}^{\mathrm{med}})+ \overline{\mu}(\mathcal{Q}^{\mathrm{high}})$. Taking $r\rightarrow \infty$ and applying Lemma \ref{small} and Lemma \ref{med} and \ref{high}, we obtain
		\[
			\mu(\mathcal{P})=\lim_{r\rightarrow \infty} \mu(\mathcal{P}_{\leq r})=\zeta_{X_0}(n+1)^{-1},
		\]
		which is what we need to show.
	\end{proof}
	\begin{proof}[Proof of Corollary \ref{generalbertinisingularity}]
		Assuming that $S$ is a finite subscheme of $Y$ as in Corollary \ref{generalbertinisingularity}, we prove as in the proof of Theorem \ref{generalbertini}, applying that for any $d\geq 0$,
		\[
			\mathcal{P}_{S,d}\subset \mathcal{P}_{S,\leq r, d}\subset \mathcal{P}_{S,d}\cup \mathcal{Q}_{r,d}^{\mathrm{med}}\cup \mathcal{Q}_{d}^{\mathrm{high}}.
		\]
		By the same argument as above, we get
		\[
			\mu(\mathcal{P}_S)=\lim_{r\rightarrow \infty} \mu(\mathcal{P}_{S,\leq r})=\zeta_{X_0}(n+1)^{-1}\cdot\prod_{x\in \lvert S\rvert} \left( 1-\#\kappa(x)^{-1}\right).
		\]
	\end{proof}

\section{Smoothing procedure}\label{smoothing}

To prove Theorem \ref{1cycle}, it's enough to prove the following proposition :

\begin{prop}\label{final1cycle}
	Let $X$ be a smooth projective variety of dimension $n\geq 2$ over $\F_q$, and $\li$ a very ample line bundle over $X$. If $Z$ is a singular closed curve on $X$, then we can find smooth closed curves $Z_1$ and $Z_2$ on $X$ such that 
	\[
		Z\sim_{\mathrm{rat}} Z_1-Z_2.
	\]
\end{prop}

In fact, any $1$-cycle on $X$ is of the form $Z=\sum_i a_iZ_i$ where each $a_i$ is an integer and each $Z_i$ is the class of a closed curve on $X$. To find a smooth $1$-cycle on $X$ rationally equivalent to $Z$, it is enough to find a smooth $1$-cycle on $X$ rationally equivalent to $Z_i$ for each singular component $Z_i$ of $Z$. Therefore Proposition \ref{final1cycle} implies Theorem \ref{1cycle}.
\\

To prove Proposition \ref{final1cycle}, we will use the following equivalent form of Theorem 1.1 in \cite{Po08} :
\begin{thm}[Theorem 1.1 in \cite{Po08}]\label{poonen}
	Let $X$ be a projective scheme of dimension $n$ over $\F_q$, and $X_0$ a smooth quasi-projective subscheme of $X$ of dimension $m$. Let $\li$ be a very ample line bundle on $X$. Let $Z$ be a closed subscheme of $X$. Assume that the scheme-theoretic intersection $Z_0=Z\cap X_0$ is smooth of dimension $l$ (set $l=-1$ if $Z_0=\emptyset$). Define
	\[
		\mathcal{P}=\bigcup_{d\geq 0}\left\{ \sigma\in \ho^0(X,\mathcal{I}_Z\otimes \li^{\otimes d})\ ;\ \mathrm{div}\sigma\cap X_0 \text{ is smooth of dimension } m-1 \right\},
	\]
	where $\mathcal{I}_Z$ is the ideal sheaf of $Z$ in $X$.
	\begin{enumerate}[label=\roman*)]
		\item If $m>2l$, then
		\[
			\mu_Z(\mathcal{P}):=\lim_{d\rightarrow \infty}\frac{\#\left(\mathcal{P}\cap \ho^0(X,\mathcal{I}_Z\otimes \li^{\otimes d})\right)}{\#\ho^0(X,\mathcal{I}_Z\otimes \li^{\otimes d})}=\zeta_{Z_0}(m-l)^{-1}\zeta_{X_0-Z_0}(m+1)^{-1}.
		\]
		In this case, when $d\gg1$, we may find a section $\sigma\in \ho^0(X,\mathcal{I}_Z\otimes \li^{\otimes d})$ such that $\mathrm{div}\sigma$ contains $Z$ and $\mathrm{div}\sigma\cap X_0$ is smooth of dimension $m-1$.
		\item If $m\leq 2l$, then $\mu_Z(\mathcal{P})=0$.
	\end{enumerate}
\end{thm}
\begin{rmq}
The original statement is the case when $X$ is a projective space $\pr^n$ and $\li$ is the sheaf $\Oc(1)$. The theorem stated in this form is implied by the original one. In fact, as $\li$ is very ample, it induces an embedding of $X$ into a projective space $\pr^N=\pr(\ho^0(X,\li)^{\vee})$, where $\li$ is the restriction of $\Oc(1)$ on $X$. To distinguish the two ideal sheaves of $Z$ in $X$ and in $\pr^N$, we denote by $\mathcal{I}'_Z$ the ideal sheaf of $Z$ in $\pr^N$. When $d\gg1$, we have
\[
	\ho^0(X,\mathcal{I}_Z\otimes \li^{\otimes d})\simeq \ho^0(\pr^N,\mathcal{I}'_Z\otimes \Oc(d)),
\]
and the proportion calculated on the two sides are equal. 
\end{rmq}

\begin{proof}[Proof of Proposition \ref{final1cycle}]
	
	Let $Z_0=Z-\mathrm{Sing}(Z)$ be the smooth part of $Z$. Set $X_0=X-\mathrm{Sing}(Z)$. 
	
	If $n=2$, the curve $Z$ determines a line bundle $\Oc(Z)$ on $X$. We may apply Corollary \ref{generalbertinisingularity} directly to the case $Y=\mathrm{Sing}(Z),S=\mathrm{Sing}(Z), \mi=\Oc(-Z)=\Oc(Z)^{\vee}$. Then when $d\gg 1$, we can always find a section $\beta\in \ho^0(X,\Oc(-Z)\otimes \li^{\otimes d})$ such that $\mathrm{div}\beta$, also denoted by $Z_{\beta}$, is a smooth curve satisfying $\mathrm{div}\beta\cap \mathrm{Sing}(Z)=\emptyset$. Let $\sigma_Z\in \ho^0(X,\Oc(Z))$ be a defining section of $Z$, i.e. $Z=\mathrm{div}\sigma_Z$. Then $\sigma:=\sigma_Z\cdot \beta\in \ho^0(X,\mathcal{I}_Z\otimes \li^{\otimes d})$ is such that 
	\[
		\mathrm{div}\sigma=Z\cup Z_\beta.
	\]
	We may assume that the $d$ we choose is large enough so that we can also find $\sigma'\in \ho^0(X,\li^{\otimes d})$ such that $\mathrm{div}\sigma'$ is a smooth curve on $X$. Then of course $\mathrm{div}\sigma'$ is rationally equivalent to the cycle $Z+Z_{\beta}$, and hence $Z$ is rationally equivalent to $\mathrm{div}\sigma'-Z_{\beta}$, where $\mathrm{div}\sigma'$ and $Z_\beta$ are both smooth. Proposition \ref{final1cycle} then follows, setting $Z_1=\mathrm{div \sigma'}$ and $Z_2=Z_\beta$.\\

	We may therefore assume $n\geq 3$. Theorem \ref{poonen} tells us that when $d_1\gg 1$, we can always find $\sigma_1\in \ho^0(X,\mathcal{I}_Z\otimes \li^{\otimes d_1})$ such that $\mathrm{div}\sigma_1$ contains $Z$ and $\mathrm{div}\sigma_1\cap X_0$ is smooth of dimension $n-1$. We may choose $d_1$ large enough so that we can also find a section $\sigma'_1\in \ho^0(X,\li^{\otimes d_1})$ whose divisor $\mathrm{div}\sigma'_1$ is smooth of dimension $n-1$ by Theorem 1.1 in \cite{Po04}. 
	
	We then continue to choose $d_2\gg 1$ such that applying Theorem \ref{poonen} to $\mathrm{div}\sigma_1\cap X_0$ instead of $X_0$, there exists $\sigma_2\in \ho^0(X,\mathcal{I}_Z\otimes \li^{\otimes d_2})$ such that $\mathrm{div}\sigma_2$ contains $Z$ and $X_0\cap \mathrm{div}\sigma_1\cap \mathrm{div}\sigma_2$ is smooth of dimension $n-2$. This $d_2$ can also be chosen large enough so that there exists $\sigma'_2\in \ho^0(X,\li^{\otimes d_2})$ such that $\mathrm{div}\sigma'_1\cap \mathrm{div}\sigma'_1$ is smooth of dimension $n-2$, also by Theorem 1.1 in \cite{Po04}.
	
	We may continue this procedure until we have chosen $d_1,\dots, d_{n-2}\in\Z_{>0}$, sections $\sigma_1,\dots,\sigma_{n-2}$ and $\sigma'_1,\dots,\sigma'_{n-2}$ where $\sigma_i\in \ho^0(X,\mathcal{I}_Z\otimes \li^{\otimes d_i})$ and $\sigma'_i\in \ho^0(X,\li^{\otimes d_i})$ such that 
	\begin{enumerate}[$\bullet$]
		\item $\mathrm{div}\sigma_1\cap\cdots\cap \mathrm{div}\sigma_{n-2}$ contains $Z$,
		\item $X_0\cap \mathrm{div}\sigma_1\cap\cdots\cap \mathrm{div}\sigma_{n-2}$ is smooth of dimension $2$, and
		\item $\mathrm{div}\sigma'_1\cap\cdots\cap \mathrm{div}\sigma'_{n-2}$ is smooth of dimension $2$.
	\end{enumerate}

	Now $\mathrm{div}\sigma_1\cap\cdots\cap \mathrm{div}\sigma_{n-2}$ is a projective surface containing $Z$ such that $X_0\cap \mathrm{div}\sigma_1\cap\cdots\cap \mathrm{div}\sigma_{n-2}$ is smooth. Denote this surface $\mathrm{div}\sigma_1\cap\cdots\cap \mathrm{div}\sigma_{n-2}$ by $S$. Then $\li|_S$ is a very ample line bundle on $S$ and $S-\mathrm{Sing}(Z)=X_0\cap S$ is smooth. As $S$ is smooth outside of finitely many points, $Z$ still determines a line bundle $\Oc(Z)$ on $S$.

	We apply Corollary \ref{generalbertinisingularity} to the surface $S$, where $X, Y, S, \li,\mi$ in the statement of the Corollary are replaced by $S, \mathrm{Sing}(Z), \mathrm{Sing}(Z), \li|_S, \Oc(-Z)$, respectively. In particular, we can choose a $d_{n-1}\gg 1$ and find a section $\beta\in\ho^0(S,\Oc(-Z)\otimes (\li^{\otimes d_{n-1}})|_S)$ such that $\mathrm{div}\beta$, also denoted by $Z_{\beta}$, is a smooth curve satisfying $\mathrm{div}\beta\cap\mathrm{Sing}(Z)=\emptyset$. Let $\sigma_Z\in \ho^0(S, \Oc(Z))$ be a defining section of $Z$ in $S$. Then $\sigma_Z\cdot\beta\in \ho^0(S,(\mathcal{I}_Z\otimes \li^{\otimes d_{n-1}})|_S)$ is a section such that
	\[
		\mathrm{div}(\sigma_Z\cdot\beta)=Z\cup Z_\beta.
	\]
	In particular, as $Z_{\beta}$ is smooth, we deduce $Z_{\beta}\not=Z$.

	Moreover, we may assume that $d_{n-1}$ is large enough so that 
	\[
		\ho^0(X,\mathcal{I}_Z\otimes \li^{\otimes d_{n-1}})\longrightarrow \ho^0(S,(\mathcal{I}_Z\otimes \li^{\otimes d_{n-1}})|_S)
	\]
	is surjective, and that we can find a $\sigma'_{n-1}\in\ho^0(X, \li^{\otimes d_{n-1}})$ such that $\mathrm{div}\sigma'_1\cap\cdots\cap \mathrm{div}\sigma'_{n-1}$ is a smooth curve in $X$. Choose a pre-image $\sigma_{n-1}$ of $\sigma_Z\cdot\beta$ in $\ho^0(X,\mathcal{I}_Z\otimes \li^{\otimes d_{n-1}})$. We then have 
	\begin{enumerate}[$\bullet$]
		\item $\mathrm{div}\sigma_1\cap\cdots\cap \mathrm{div}\sigma_{n-1}=Z\cup Z_{\beta}$ where $Z_{\beta}$ is a smooth curve,
		\item $Z_{\beta}\cap\mathrm{Sing}(Z)=\emptyset$, and
		\item $\mathrm{div}\sigma'_1\cap\cdots\cap \mathrm{div}\sigma'_{n-1}$ is smooth of dimension $1$.
	\end{enumerate}
	Set $\mathrm{div}\sigma'_1\cap\cdots\cap \mathrm{div}\sigma'_{n-1}=Z_1$ and $Z_{\beta}=Z_2$. As for each $1\leq i\leq n-1$, $\sigma_i$ and $\sigma'_i$ are both global sections in $\ho^0(X, \li^{\otimes d_i})$, $Z+Z_2$ is rationally equivalent to $Z_1$ as algebraic $1$-cycles. So $Z$ is rationally equivalent to $Z_1-Z_2$ and the conclusion of Proposition \ref{final1cycle} follows.
\end{proof}


\end{document}